\documentclass{amsart}
\usepackage{amsmath,amssymb,a4wide,color}
\usepackage{latexsym}
\usepackage{auto-pst-pdf} 
\usepackage{pstricks,pst-node}

\usepackage{xspace}

\newtheorem{theorem}{Theorem}[section]
\newtheorem{proposition}[theorem]{Proposition}
\newtheorem{corollary}[theorem]{Corollary}

\newtheorem{lemma}[theorem]{Lemma}
\newtheorem{example}[theorem]{Example}

\def\v{\vert}
\def\si{\sigma}
\def\a{\ensuremath{\mathcal A}\xspace}

\def\s{\ensuremath{\mathcal S}\xspace}
\def\q{\ensuremath{\mathcal Q}\xspace}
\def\t{\ensuremath{\mathcal T}\xspace}
\def\o{\ensuremath{\mathcal O}\xspace}
\def\fc{\ensuremath{\mathcal F}\xspace}

\newcommand{\asc}{{\rm asc\,}}
\newcommand{\plat}{{\rm plat\,}}
\newcommand{\des}{{\rm des\,}}
\newcommand{\mqn}{\mathcal{Q}_n}
\newcommand{\msn}{\mathfrak{S}_n}
\newcommand{\mssn}{\mathfrak{S}_3}

\title{Restricted Stirling permutations}
\author[D.~Callan]{David Callan}
\address{Department of Statistics,
        University of Wisconsin-Madison,
        Madison, WI \ 53706-1532, U.S.A.}
\email{callan@stat.wisc.edu (D. Callan)}
\author[S.-M.~Ma]{Shi-Mei~Ma}
\address{School of Mathematics and Statistics,
        Northeastern University at Qinhuangdao,
         Hebei 066004, P.R. China}
\email{shimeimapapers@163.com (S.-M. Ma)}
\author[T.~Mansour]{Toufik Mansour}
\address{Department of Mathematics, University of Haifa, 3498838 Haifa, Israel}
\email{toufik@math.haifa.ac.il (T. Mansour)}
\begin{document}
\subjclass[2010]{05A05; 05A15; 40C15}
\keywords{Stirling permutations; Plateaus; Pattern avoidance}

\begin{abstract}
In this paper, we study the generating functions for the number of pattern restricted Stirling
permutations with a given number of plateaus, descents and ascents. Properties of the generating
functions, including symmetric properties and explicit formulas are studied. Combinatorial explanations are given for some equidistributions.
\end{abstract}

\maketitle

\section{Introduction and main results}
Let $\msn$ denote the symmetric group of all permutations of $[n]$, where $[n]=\{1,2,\ldots,n\}$.
A permutation $\sigma=\sigma_1\sigma_2\cdots\sigma_n\in\msn$ is said to {\it contain} another permutation $\tau=\tau_1\tau_2\cdots\tau_k\in \mathfrak{S}_k$ as a {\it pattern} if $\sigma$ has
a subsequence order-isomorphic to $\tau$, where $n\geq k$. If there is no such subsequence, then we say that $\sigma$ {\it avoids} the pattern $\tau$.
Pattern avoidance was first studied by Knuth~\cite{Knuth73} and he found that, for $\tau\in\mssn$,
the number of permutations in $\msn$ avoiding
$\tau$ is given by the $n$th Catalan number. Later, Simion and Schmidt~\cite{Simion85} determined the
number of permutations in $\msn$ simultaneously avoiding any given set of patterns $\tau\in\mssn$.
From then on, there has been a large literature devoted to this topic, see~\cite{Bona12,Mansour09} for instance.

Stirling permutations were introduced by Gessel and Stanley \cite{Gessel78}. A {\it Stirling permutation} of order $n$ is a permutation $\sigma$ of the multiset $\{1,1,2,2,\ldots,n,n\}$ such that
every element between the two occurrences of $i$ is greater than $i$ for each $i\in[n]$.
Denote by $\mqn$ the set of Stirling permutations of order $n$. Let $\sigma=\sigma_1\sigma_2\cdots\sigma_{2n-1}\sigma_{2n}\in\mqn$. Throughout this paper, we always let
\begin{align*}
&\des(\sigma)=\#\{i\mid 1\leq i\leq 2n-1\mbox{ and }\sigma_i>\sigma_{i+1}\},\\
&\asc(\sigma)=\#\{i\mid 1\leq i\leq 2n-1\mbox{ and }\sigma_i<\sigma_{i+1}\},\\
&\plat(\sigma)=\#\{i\mid 1\leq i\leq 2n-1\mbox{ and }\sigma_i=\sigma_{i+1}\}.
\end{align*}
denote the number of descents, ascents and plateaus of $\sigma$, respectively. Then the equations
$$C_n(x)=\sum_{\sigma\in\mqn}x^{\des(\sigma)+1}=\sum_{i=1}^nC(n,k)x^k$$
define the {\it second-order Eulerian polynomials} $C_n(x)$ and the {\it second-order Eulerian numbers} $C(n,k)$.
Let $\asc(\sigma)+1$ and $\des(\sigma)+1$ be the number of {\it augmented ascents} and {\it augmented descents} of
$\sigma$, respectively, that is, the number of ascents and descents when $\sigma$ is augmented with a 0 at the start and end. B\'ona \cite[Proposition~1]{Bona08} proved that the augmented ascents, augmented descents and plateaus are equidistributed over the set $\mqn$. Let $$C_n(p,q,r)=\sum_{\sigma\in\mqn}p^{\plat(\sigma)}q^{\des(\sigma)}r^{\asc(\sigma)}.$$
Janson~\cite[Theorem 2.1]{Janson08} discovered that the trivariate generating function
$qrC_n(p,q,r)$ is symmetric in $p,q,r$, which implies B\'ona's equidistributed result.

The notion of pattern avoidance can be extended to Stirling permutations in a straightforward way.
We say that $\sigma\in\mqn$ contains the pattern $\tau=\tau_1\tau_2\cdots\tau_k$ if for some
$1\leq i_1<i_2<\cdots<i_k\leq 2n$, we have $\sigma_{i_s}<\sigma_{i_t}$ whenever $\tau_s<\tau_t$. A Stirling
permutation is said to avoid any pattern it does not contain.
Let $\mathcal{Q}_{n}(\tau)$ be the set of Stirling permutations of order $n$ avoiding the pattern $\tau$.
Recently, Kuba and Panholzer~\cite{Kuba12} obtained enumerative formulas for Stirling permutations avoiding a set of patterns of length three. For example, it follows from~\cite[Theorem~1]{Kuba12} that
\begin{equation*}\label{Kuba-formula}
\begin{split}
\#\mathcal{Q}_{n}(213)&=\frac{1}{2n+1}\binom{3n}{n},\\
\#\mathcal{Q}_{n}(123)&=\mathcal{Q}_{n}(132)=\sum_{j=0}^n\frac{\binom{n}{j}\binom{n+j-1}{n-j}}{n+1-j}.
\end{split}
\end{equation*}
We denote the generating function for the number of Stirling permutations of order $n$ according to the number plateaus, descents and ascents by
$$C_{n,\tau}(p,q,r)=\sum_{\sigma\in\mqn(\tau)}p^{\plat(\sigma)}q^{\des(\sigma)}r^{\asc(\sigma)}.$$
We now present the three main results of this paper.

\begin{theorem}\label{mainthm01}
For $n\geq 1$, the generating function $qrC_{n,213}(p,q,r)$ is symmetric in $p,q,r$. Furthermore,
the number of Stirling permutations in $\mqn(213)$ with exactly $m$ ascents, $d$ descents and $k$ plateaus is given by
\begin{align}\label{eq04}
\left\{\begin{array}{ll}
\frac{1}{n}\binom{n}{m+1}\binom{n}{d+1}\binom{n}{k},&\mbox{ if }2n-1=m+d+k\\
0,&\mbox{otherwise}.\end{array}\right.
\end{align}
Moreover,
\begin{align}
\sum_{\sigma\in \mqn(213)}p^{\plat(\sigma)}=\frac{1}{n}\sum_{i=0}^{n-1}\binom{n}{i}\binom{2n}{n-1-i}p^{n-i}.\label{eq05}
\end{align}
\end{theorem}

\begin{theorem}\label{mainthm02}
For $n\geq 1$, the generating function $qC_{n,123}(p,q,r)$ is symmetric in $p,q$. Furthermore,
\begin{align}
\sum_{\sigma\in\mqn(123)}p^{\plat(\sigma)}=\frac{1}{n+1}\sum_{j=0}^{n}\binom{n+1}{j}\binom{2n-j}{n+j}p^{n-j}.\label{eq050}
\end{align}
\end{theorem}

The symmetric properties of $qrC_{n,213}(p,q,r)$ and $qC_{n,123}(p,q,r)$ lead to the following corollary.

\begin{corollary}
For all $n\geq0$,
\begin{align}
&\sum_{\sigma\in \mqn(123)}q^{\des(\sigma)+1}=\sum_{\sigma\in \mqn(123)}q^{\plat(\sigma)}.\label{eq01}\\
&\sum_{\sigma\in \mqn(213)}q^{\asc(\sigma)+1}=\sum_{\sigma\in \mqn(213)}q^{\des(\sigma)+1}=\sum_{\sigma\in \mqn(213)}q^{\plat(\sigma)}.\label{eq02}
\end{align}
\end{corollary}

\begin{theorem}\label{mainthm00}
For $n\geq0$,
$$\sum_{\sigma\in \mqn(132)}q^{\plat(\sigma)}=\sum_{\sigma\in \mqn(123)}q^{\plat(\sigma)}.$$
Moreover, the number of $132$-avoiding Stirling permutations of order $n$ with exactly $d$ descents is given by
$$\frac{\binom{n-1}{d}}{n+1}\sum_{j=0}^{n+1}\binom{n+1}{j}\binom{j}{d+1-j}.$$
\end{theorem}

\section{Analytic proofs of the main theorems}
In this section, we present Analytic proofs of Theorems \ref{mainthm01}, \ref{mainthm02} and \ref{mainthm00}. More precisely, we find explicit formulas for the generating function $\sum_{n\geq0}C_{n,\tau}(p,q,r)x^n$ for $\tau\in \mathfrak{S}_3$. Since the reversal operation ($\sigma_1\sigma_2\cdots\sigma_{2n}\mapsto\sigma_{2n}\cdots\sigma_2\sigma_1$) preserves the set of Stirling permutations, we only need to consider the three cases, $\tau=123$, $\tau=132$,  $\tau=213$. For the latter case, we use the block decompositions technique (for instance, see \cite{MV}), while for the former two cases, we use the kernel method (for instance, see \cite{HouM}).

\subsection{The case $213$}\label{subsec213}
Define $$C_{213}(x,p,q,r)=\sum_{n\geq0}\sum_{\sigma\in \mqn(213)}x^np^{\plat(\sigma)}q^{\des(\sigma)}r^{\asc(\sigma)}.$$
Note that each nonempty Stirling permutation $\sigma$ that avoids $213$ can be represented as $\sigma'1\sigma''1\sigma'''$ such that
\begin{itemize}
\item each letter of $\sigma'$ is greater than each letter of $\sigma''$;
\item each letter of $\sigma''$ is greater than each letter of $\sigma'''$;
\item $\sigma',\sigma'',\sigma'''$ are Stirling permutations that avoid $213$.
\end{itemize}
Hence, by considering the 8 possibilities where one of $\sigma',\sigma'',\sigma'''$ is empty or not, we obtain that the generating function $C_{213}(x,p,q,r)$ satisfies
\begin{align*}
C_{213}(x,p,q,r)&=1+xp+x(pr+qr+pq)(C_{213}(x,p,q,r)-1)\\
&+xqr(r+p+q)(C_{213}(x,p,q,r)-1)^2+xq^2r^2(C_{213}(x,p,q,r)-1)^3,
\end{align*}
which leads to the following result.
\begin{theorem}\label{lem123a4}
The generating function $f=C_{213}(x,p,q,r)-1$ satisfies
$$f=xp+x(pr+qr+pq)f+xqr(r+p+q)f^2+xq^2r^2f^3.$$
\end{theorem}

\subsubsection{Proof of Theorem \ref{mainthm01}}\label{subsec13}
Theorem~\ref{lem123a4} shows that the generating function $$g=qr(C_{213}(x,p,q,r)-1)$$ satisfies
$g=x(p+g)(q+g)(r+g)$. Thus, the generating function $g$ is symmetric in $p,q,r$.
Moreover, by Lagrange Inversion Formula, we obtain that the coefficient of $x^n$ in $g$ is given by
\begin{align*}
[x^n]g&=\frac{1}{n}[y^{n-1}](p+y)^n(q+y)^n(r+y)^n\\
&=\frac{1}{n}\sum_{i=0}^n\sum_{j=0}^n\binom{n}{i}\binom{n}{j}\binom{n}{i+j+1}q^{n-i}r^{n-j}p^{i+j+1},
\end{align*}
which completes the proof of \eqref{eq04}.

If we let $g=C_{213}(x,p,1,1)$, then Theorem~\ref{lem123a4} gives $g=1+x(p-1+g)g^2$. Thus, by Lagrange Inversion Formula, we obtain that the coefficient of $x^n$ in $g$ is given by
\begin{align*}
[x^n]g&=\frac{1}{n}[y^{n-1}](p+y)^n(y+1)^{2n}=\frac{1}{n}\sum_{i=0}^{n-1}\binom{n}{i}\binom{2n}{n-1-i}p^{n-i}.
\end{align*}
Hence, the number of Stirling permutations in $\mqn(213)$ with exactly $k$ plateaus is given by $\frac{1}{n}\binom{n}{k}\binom{2n}{k-1}$,
which completes the proof.\hfill$\Box$

\subsection{The case $123$}
For short notation, we define $f(n)=C_{n,123}(p,q,r)$. Conditioning on the initial entries of
permutations, we define $$f(n|i_1i_2\cdots i_s)=\sum_{\sigma=i_1i_2\cdots i_s\sigma'\in\mqn(123)}x^np^{\plat(\sigma)}q^{\des(\sigma)}r^{\asc(\sigma)}.$$
\begin{lemma}\label{f123a1}
For all $n\geq2$,
$$f(n)-pqf(n-1)=\sum_{i=1}^nf(n|ii)+q(r-p)\sum_{i=1}^{n-1}f(n-1|ii).$$
\end{lemma}
\begin{proof}
Clearly, $f(n)=\sum_{i=1}^nf(n|i)$, and
$f(n|i)=f(n|ii)+f(n|inn)$
(when $i<n$).
Thus, we obtain
\begin{align}\label{eqpqrst2}
f(n)=\sum_{i=1}^nf(n|ii)+\sum_{i=1}^{n-1}f(n|inn).
\end{align}
On the other hand, for all $n\geq2$, we have
\begin{align}
f(n|inn)&=f(n|inni)+f(n|inn(n-1)(n-1))\notag\\
&=qrf(n-1|ii)+pqf(n-1|i(n-1)(n-1)),\label{eqpqrst4}
\end{align}
which, by \eqref{eqpqrst2} and \eqref{eqpqrst4}, implies the required result.
\end{proof}

\begin{lemma}\label{f123a2}
For $1\leq i\leq n-2$ and $n\geq4$,
\begin{align*}
&f(n|ii)-2pqf(n-1|ii)+p^2q^2f(n-2|ii)\\
&=pq\sum_{j=1}^{i-1}f(n-1|jj)+pq(pr+qr-2pq)\sum_{j=1}^{i-1}f(n-2|jj)+p^2q^2(r-p)(r-q)\sum_{j=1}^{i-1}f(n-3|jj).
\end{align*}
Moreover, $f(n|nn)=pqf(n-1)$ and $f(n|(n-1)(n-1))=pqf(n-1)+p^2q(r-q)f(n-2)$.
\end{lemma}
\begin{proof}
By the definitions, we have $f(n|nn)=pqf(n-1)$. Thus,
$$f(n|(n-1)(n-1))=\sum_{j=1}^{n-2}f(n|(n-1)(n-1)j)+f(n|(n-1)(n-1)nn),$$
which implies
$$f(n|(n-1)(n-1))=pq\sum_{j=1}^{n-2}f(n-1|j)+prf(n-1|(n-1)(n-1))=pqf(n-1)+p^2q(r-q)f(n-2).$$
Now, let $1\leq i\leq n-2$, then
\begin{align*}
f(n|ii)&=\sum_{j=1}^{i-1}f(n|iij)+f(n|iinn)=pq\sum_{j=1}^{i-1}f(n-1|j)+f(n|iinn).
\end{align*}
Similarly,
$$f(n|iinn)=pqf(n-1|ii(n-1)(n-1))+p^2qr\sum_{j=1}^{i-1}f(n-2|j),$$
which implies
$$f(n|iinn)-pqf(n|ii(n-1)(n-1))=p^2qr\sum_{j=1}^{i-1}f(n-2|j).$$
Thus,
\begin{align}
f(n|ii)-pqf(n-1|ii)
&=pq\sum_{j=1}^{i-1}f(n-1|j)+p^2q(r-p)\sum_{j=1}^{i-1}f(n-2|j).\label{eqf123d1}
\end{align}
On the other hand, by Lemma \ref{f123a1}, we have that
$$\sum_{j=1}^{i-1}f(n|j)-pq\sum_{j=1}^{i-1}f(n-1|j)=\sum_{j=1}^{i-1}f(n|jj)+q(r-p)\sum_{j=1}^{i-1}f(n-1|jj).$$
Hence, by using \eqref{eqf123d1}, we complete the proof.
\end{proof}

Define $L_n(v)=\sum_{i=1}^nf(n|ii)v^{i-1}$. We now present the following result.
\begin{proposition}\label{fp123a1}
For all $n\geq4$,
\begin{align*}
&L_n(v)-2pqL_{n-1}(v)+p^2q^2L_{n-2}(v)\\
&\qquad=pqf(n-1)v^{n-2}(1+v)+p^2(rq-3q^2)f(n-2)v^{n-2}+p^2q^2f(n-2)v^{n-2}\\
&\qquad+\frac{pqv}{1-v}(L_{n-1}(v)-v^{n-3}L_{n-1}(1))+\frac{pq(qr+pr-2pq)v}{1-v}(L_{n-2}(v)-v^{n-3}L_{n-2}(1))\\
&\qquad+\frac{p^2q^2(r-p)(r-q))v}{1-v}(L_{n-3}(v)-v^{n-3}L_{n-3}(1)),\\
&f(n)-pqf(n-1)=L_n(1)+q(r-p)L_{n-1}(1),
\end{align*}
where $L_1(v)=p$, $L_2(v)=p^2(r+qv)$, $L_3(v)=p^3qr+p^2qr(2p+q)v+p^2q(pr+qr+pq)v^2$, $f(0)=1$, $f(1)=p$ and $f(2)=p(pq+pr+qr)$.
\end{proposition}
\begin{proof}
The initial conditions can be obtained from the definitions. The recurrence relation for $L_n(v)$ is obtained by multiplying the recurrence relation for $f(n|ii)$ in Lemma \ref{f123a2} by $v^{i-1}$ and summing over $i=1,2,\ldots,n-2$. The recurrence relation for $f(n)$ follows immediately from Lemma \ref{f123a1}.
\end{proof}

Define $L(x;v)=\sum_{n\geq1}L_n(v)x^n$ and let $F(x)=C_{123}(x,p,q,r)$ (for short notation).
By multiplying the first recurrence in Proposition \ref{fp123a1} by $x^n$ and summing over $n\geq4$, we obtain
\begin{align*}
&L(x;v)-L_1(v)x-L_2(v)x^2-L_3(v)x^3-2pqx(L(x,v)-L_1(v)x-L_2(v)x^2)+p^2q^2x^2(L(x,v)-L_1(v)x)\\
&=pqx(F(xv)-1-f(1)xv-f(2)x^2v^2)+\frac{pqx}{v}(F(xv)-1-f(1)xv-f(2)x^2v^2)\\
&+p^2q(r-q)x^2(F(xv)-1-f(1)xv)-p^2q^2x^2(F(xv)-1-f(1)xv)\\
&+\frac{pqvx}{1-v}(L(x,v)-L_1(v)x-L_2(v)x^2-\frac{1}{v^2}(L(xv,1)-L_1(1)xv-L_2(1)x^2v^2))\\
&+\frac{pq(qr+pr-2pq)x^2v}{1-v}(L(x,v)-L_1(v)x-\frac{1}{v}(L(xv,1)-L_1(1)xv))\\
&+\frac{p^2q^2(r-p)(r-q)x^3v}{1-v}(L(x,v)-L(xv,1)),
\end{align*}
which, by several simple algebraic operations, implies
\begin{align}
&\left((1-pqx)^2-\frac{pqxv(1-(p-r)qx)(1-p(q-r)x)}{1-v}\right)L(x;v)\notag\\
&\qquad\quad=px(1-pqx)(1+xp(r-q))-\frac{pqx(1+pxv(r-q))(1+qxv(r-p))}{v(1-v)}L(xv;1)\label{eqLL01}\\
&\qquad\quad+\frac{pqx(1+v+pxv(r-2q))}{v}(F(xv)-1).\notag
\end{align}
By multiplying the second recurrence in Proposition \ref{fp123a1} by $x^n$ and summing over $n\geq2$, we obtain
\begin{align}
&(1-pqx)(F(x)-1)=(1+q(r-p)x)L(x;1).\label{eqLL02}
\end{align}
By finding $L(x;1)$ from \eqref{eqLL02} and using it to simplify \eqref{eqLL01}, we obtain the following result.

\begin{theorem}\label{th1f123}
The generating function $C_{123}(x,p,q,r)$ is given by
$$C_{123}(x,p,q,r)=1+\frac{1+q(r-p)x}{1-pqx}L(x;1),$$
where the generating function $L(x;v)$ satisfies
\begin{align*}
&\left((1-pqx)^2-\frac{pqxv(1-(p-r)qx)(1-p(q-r)x)}{1-v}\right)L(x;v)\\
&\qquad\quad\qquad\quad=\frac{px(1-pqx)(1+xp(r-q))(1-v(1-q))}{1-v}\left(1-(1-q)v-qvF(xv)\right).
\end{align*}
\end{theorem}
Theorem~\ref{th1f123} gives
\begin{align}
&\left((1-pqx/v)^2-\frac{pqx(1-(p-r)qx/v)(1-p(q-r)x/v)}{1-v}\right)L(x/v;v)\notag\\
&\qquad\quad\quad=\frac{px/v(1-pqx/v)(1+xp(r-q)/v)(1-v(1-q))}{1-v}\left(1-(1-q)v-qvF(x)\right),\label{th1f123eq1}
\end{align}
where $F(x)=1+\frac{(1+q(r-p)x)}{1-pqx}L(x;1)=C_{123}(x,p,q,r)$.
This type of functional equation can be solved systematically using the kernel method (see \cite{HouM} and references therein). In order to do that, we define
$$K(v)=(1-pqx/v)^2-\frac{pqx(1-(p-r)qx/v)(1-p(q-r)x/v)}{1-v}.$$
So, if we assume that $v=v_0=v_0(x,p,q,r)$ in Theorem \ref{th1f123} (we shall show that $v_0$ is the solution) such that $K(v_0)=0$, then \eqref{th1f123eq1} gives
\begin{align}
\frac{px/v_0(1-pqx/v_0)(1+xp(r-q)/v_0)(1-v_0(1-q))}{1-v_0}\left(1-(1-q)v_0-qv_0F(x)\right)=0,
\end{align}
which implies
\begin{align}
F(x)=C_{123}(x,p,q,r)=1+\frac{1-v_0}{qv_0},
\end{align}
where $v_0$ satisfies $K(v_0)=0$, that is,
$$-p^2q^2x^2(1-x(r-q)(r-p))+pqx(2+x((p+q)r-pq))v_0-(1+pqx)v_0^2+v_0^3=0.$$
If we set $f=qC_{123}(x,p,q,r)-q+1$, then $f=q+\frac{1-v_0}{v_0}-q+1=\frac{1}{v_0}$, which implies
$$-p^2q^2x^2(1-x(r-q)(r-p))f^3+pqx(2+x((p+q)r-pq))f^2-(1+pqx)f+1=0.$$
Hence, we can state the following result.

\begin{theorem}\label{th2f123}
The generating function $f=qC_{123}(x,p,q,r)-q+1$ satisfies
\begin{align*}
f=1+pqx(-1+(2+x(pr+qr-pq))f-pqx(1-x(p-r)(q-r))f^2)f.
\end{align*}
\end{theorem}

\subsubsection{Proof of Theorem \ref{mainthm02}}
Let $f=q(C_{123}(x,p,q,r)-1)+1$. Then, Theorem \ref{th2f123} can be written as
$$f=\frac{1-p^2q^2x^2(1-x(r-q)(r-p))f^3}{1+pqx-pqx(2+x(qr+pr-pq))f},$$
which shows that the generating function $f$ is symmetric in $p,q$.

Now, assume that $g=xC_{123}(x,p,1,1)$. Then, Theorem \ref{th2f123} gives that the generating function $g$ satisfies
\begin{align}
g =\frac{x(1-pg+pg^2)}{(1-pg)^2}.\label{eqg123g1}
\end{align}
Then, by Lagrange Inversion Formula, we have that the coefficient of $x^n$ in $g$ is given by
$$[x^n]g=\frac{[y^{n-1}]}{n}\sum_{j=0}^n\binom{n}{j}\frac{p^jy^{2j}}{(1-py)^{n+j}},$$
which implies
\begin{align*}
[x^n]g&=\frac{[y^{n-1}]}{n}\sum_{j=0}^n\sum_{i\geq0}\binom{n}{j}\binom{n-1+j+i}{i}p^{j+i}y^{2j+i}\\
&=\frac{1}{n}\sum_{j=0}^{n}\binom{n}{j}\binom{2n-2-j}{n-1-2j}p^{n-1-j}.
\end{align*}
Hence, by the fact that  $C_{123}(x,p,1,1)=g/x$, we obtain that the generating function for the number of Stirling permutations of length $n$ that avoid $123$ according to the number plateaus is given by
$\frac{1}{n+1}\sum_{j=0}^{n}\binom{n+1}{j}\binom{2n-j}{n+j}p^{n-j}$.
Moreover, the number of Stirling permutations of length $n$ that avoid $123$ with exactly $k$ plateaus is given by
$\frac{1}{n+1}\binom{n+1}{k+1}\binom{n+k}{2n-k}$,
which proves \eqref{eq050}.\hfill$\Box$

\subsection{The case $132$}\label{subsec132}
Define $g(n)=C_{n,132}(p,q,r)$ and again use the notation
$$g(n|i_1i_2\cdots i_s)=\sum_{\sigma=i_1i_2\cdots i_s\sigma'\in\mqn(132)}x^np^{\textrm{plat}(\sigma)}q^{\textrm{des}(\sigma)}r^{\textrm{asc}(\sigma)}.$$
\begin{lemma}\label{g132a1}
For all $n\geq2$,
$$g(n)-prg(n-1)=\sum_{i=1}^ng(n|ii)+r(q-p)\sum_{i=1}^{n-1}g(n-1|ii).$$
\end{lemma}
\begin{proof}
Clearly, $g(n)=\sum_{i=1}^ng(n|i)$, where
$g(n|i)=g(n|ii)+g(n|i(i+1))$.
Thus, we obtain
\begin{align}\label{eq132pqrst2}
g(n)=\sum_{i=1}^ng(n|ii)+\sum_{i=1}^{n-1}g(n|i(i+1)).
\end{align}
On the other hand, by the definitions, for all $n\geq2$, we have
\begin{align}
g(n|i(i+1))&=g(n|i(i+1)(i+1))=g(n|i(i+1)(i+1)i)+g(n|i(i+1)(i+1)(i+2))\notag\\
&=qrg(n-1|ii)+prg(n-1|i(i+1)),\label{eq132pqrst4}
\end{align}
which, by \eqref{eq132pqrst2} and \eqref{eq132pqrst4}, implies the required result.
\end{proof}

\begin{lemma}\label{g132a2}
For $1\leq i\leq n-1$ and $n\geq3$,
\begin{align*}
g(n|ii)-2prg(n-1|ii)&=pq\sum_{j=1}^{i-1}g(n-1|jj)+pqr(q-p)\sum_{j=1}^{i-1}g(n-2|jj)-p^2r^2g(n-2|ii)
\end{align*}
with $g(n|nn)=pqg(n-1)$.
\end{lemma}
\begin{proof}
By the definition $g(n|nn)=pqg(n-1)$. Let $1\leq i\leq n-1$, then
\begin{align*}
g(n|ii)&=\sum_{j=1}^{i-1}g(n|iij)+g(n|ii(i+1)(i+1))=pq\sum_{j=1}^{i-1}g(n-1|j)+prg(n-1|ii),
\end{align*}
which, by $g(n|i)=g(n|ii)+g(n|i(i+1))$, implies
\begin{align*}
g(n|ii)&=pq\sum_{j=1}^{i-1}g(n-1|jj)+pq\sum_{j=1}^{i-1}g(n-1|j(j+1))+prg(n-1|ii).
\end{align*}
Thus, by \eqref{eq132pqrst4}, we obtain
\begin{align*}
g(n|ii)-prg(n-1|ii)&=pq\sum_{j=1}^{i-1}g(n-1|jj)+pqr(q-p)\sum_{j=1}^{i-1}g(n-2|jj)\\
&+prg(n-1|ii)-p^2r^2g(n-2|ii).
\end{align*}
as required.
\end{proof}

\begin{proposition}\label{gp132a1}
Define $L_n(v)=\sum_{i=1}^ng(n|ii)v^{i-1}$. For all $n\geq3$,
\begin{align*}
L_n(v)-pqv^{n-1}g(n-1)-2prL_{n-1}(v)&=\frac{pqv}{1-v}(L_{n-1}(v)-v^{n-2}L_{n-1}(1))\\
&+\frac{pqr(q-p)v}{1-v}(L_{n-2}(v)-v^{n-2}L_{n-2}(1))-p^2r^2L_{n-2}(v),\\
g(n)-prg(n-1)&=L_n(1)+r(q-p)L_{n-1}(1),
\end{align*}
where $L_1(v)=p$ and $L_2(v)=p^2(r+qv)$, $g(0)=1$, $g(1)=p$ and $g(2)=p(pr+qr+pq)$.
\end{proposition}
\begin{proof}
The initial conditions can be obtained from the definitions. By Lemma \ref{g132a1}, we have $g(n)-prg(n-1)=L_n(1)+r(q-p)L_{n-1}(1)$.
By multiplying the recurrence relation in statement of Lemma \ref{g132a2} by $v^{i-1}$ and summing over $i=1,2,\ldots,n-1$, we obtain the recurrence relation for $L_n(v)$.
\end{proof}

Define $L(x;v)=\sum_{n\geq1}L_n(v)x^n$ and let $F(x)=C_{132}(x,p,q,r)$ (for short notation).
By multiplying the recurrences in Proposition \ref{gp132a1} by $x^n$ and summing over $n\geq3$, we obtain
\begin{align*}
&\left((1-prx)^2-\frac{pqvx(1+r(q-p)x)}{1-v}\right)L(x;v)\\
&\qquad\qquad\qquad=px(1-prx)+pqx(F(xv)-1)-\frac{pqx}{1-v}\left(1+r(q-p)vx\right)L(xv;1),\\
&(1-prx)F(x)=(1+r(q-p)x)L(x;1)+1-prx.
\end{align*}
Hence, we can state the following result.
\begin{theorem}\label{th1g132}
The generating function $C_{132}(x,p,q,r)$ is given by
$$C_{132}(x,p,q,r)=1+\frac{1+r(q-p)x}{1-prx}L(x;1),$$
where the generating function $L(x;v)$ satisfies
\begin{align*}
&\left((1-prx)^2-\frac{pqvx(1+r(q-p)x)}{1-v}\right)L(x;v)\\
&\qquad\qquad=px(1-prx)-\frac{pqxv(1-prx)(1+(q-p)rxv)}{(1-prxv)(1-v)}L(xv;1).
\end{align*}
\end{theorem}

Theorem \ref{th1g132} gives
\begin{align}
&\left((1-prx/v)^2-\frac{pqx(1+r(q-p)x/v)}{1-v}\right)L(x/v;v)\\
&\qquad=px(1-prx/v)/v-\frac{pqx(1-prx/v)(1+(q-p)rx)}{(1-prx)(1-v)}L(x;1),\label{th1f132eq1}
\end{align}
where $1+\frac{1+(q-p)rx}{1-prx}L(x;1)=C_{132}(x,p,q,r)$.
This type of functional equation can be solved systematically using the kernel method (see \cite{HouM} and references therein). In order to do that, we define
$$K(v)=(1-prx/v)^2-\frac{pqx(1+r(q-p)x/v)}{1-v}.$$
So, if we assume that $v=v_0=v_0(x,p,q,r)$ in \eqref{th1f132eq1} (we shall show that $v_0$ is the solution) such that $K(v_0)=0$, then \eqref{th1f132eq1} gives
$$L(x,1)=\frac{1-v_0}{qv_0}\frac{1-prx}{1+(q-p)rx}$$
and
$$C_{132}(x,p,q,r)=1+\frac{1-v_0}{qv_0},$$
where $v_0$ satisfies
$$-p^2r^2x^2+rpx(prx-pqx+q^2x+2)v_0-(1+2prx+pqx)v_0^2+v_0^3=0.$$
So $f=qC_{132}(x,p,q,r)-q+1=\frac{1}{v_0}$, which implies
$$-p^2r^2x^2f^3+rpx(prx-pqx+q^2x+2)f^2-(1+2prx+pqx)f+1=0.$$
Hence, we can state the following result.
\begin{theorem}\label{th2g132}
The generating function $f=qC_{132}(x,p,q,r)-q+1$ satisfies
\begin{align*}
f=1+px\left(q-2r+r(2+(pr-pq+q^2)x)f-pr^2xf^2\right)f.
\end{align*}
\end{theorem}

\subsubsection{Proof of Theorem \ref{mainthm00}} Let $h=xC_{132}(x,p,1,1)$. Then, Theorem \ref{th2g132} gives
$$h=\frac{x(1-ph+ph^2)}{(1-ph)^2},$$
which, by \eqref{eqg123g1}, proves that the number of $132$-avoiding Stirling permutations of order $n$ with exactly $k$ plateaus is the same as the number of $123$-avoiding Stirling permutations of order $n$ with exactly $k$ plateaus.

If we set $h=x(qC_{132}(x,1,q,1)-q+1)$, then Theorem \ref{th2g132} gives
\begin{align*}
h=\frac{x((1-h)^2+qh(1-h)+q^2h^2)}{(1-h)^2}.
\end{align*}
Then, by Lagrange Inversion Formula, we have that the coefficient of $x^n$ in $h$ is given by
$$[x^n]h=\frac{1}{n}[y^{n-1}]\left(1+q\frac{y}{1-y}+q^2\frac{y^2}{(1-y)^2}\right)^n,$$
which implies
$$[x^n]h=\frac{1}{n}\sum_{j=0}^n\sum_{i=0}^j\binom{n}{j}\binom{j}{i}\binom{n-2}{j+i-1}q^{j+i},$$
Thus, the number of  $132$-avoiding Stirling permutations of order $n$ with exactly $d$ descents is given by
$$[x^{n+1}q^{d+1}]h=\frac{\binom{n-1}{d}}{n+1}\sum_{j=0}^{n+1}\binom{n+1}{j}\binom{j}{d+1-j},$$
as required. \hfill$\Box$
\bigskip

Finally, we note that the generating function $h=xC_{132}(x,1,1,r)$ satisfies (see Theorem \ref{th2g132})
\begin{align*}
h=\frac{x(1+(1-2r)h+rh^2)}{(1-rh)^2},
\end{align*}
which, by Lagrange Inversion Formula, implies that the coefficient of $x^n$ in $h$ is given by
$$[x^n]h=\frac{1}{n}[y^{n-1}]\sum_{\ell\geq0}\sum_{j=0}^n\sum_{i=0}^j\binom{n}{j}\binom{j}{i}\binom{2n-1+\ell}{\ell}
r^{i+\ell}(1-2r)^{j-i}y^{j+i+\ell}.$$
Thus,  the generating function for the number of  $132$-avoiding Stirling permutations of order $n$ according to the number of ascents is given by
$$[x^{n+1}]h=\frac{1}{n+1}\sum_{j=0}^{n+1}\sum_{i=0}^j\binom{n+1}{j}\binom{j}{i}\binom{3n+1-j-i}{2n+1}r^{n-1-j}(1-2r)^{j-i}.$$

\section{Some combinatorial explanations}
\subsection{The case $213$} \label{comb213} The symmetry of $qrC_{n,213}(p, q, r)$ in $p,q$ and $r$ follows from a natural bijection $\varphi: \q_{n}(213)\mapsto \t_{n-1}$, where $\t_n$ is the set of $n$-edge ternary trees. To define $\varphi$, recall that each nonempty Stirling 213-avoider $\si$ is uniquely expressible as $\si' 1 \si'' 1 \si'''$ with $\si'>\si''>\si'''$ and $\si',\si'',\si'''$ all 213-avoiders. We define $\varphi$ recursively in 8 cases according as each of $\si',\si'',\si'''$ is empty or not. First, $\varphi(11)=\epsilon$, the empty ternary tree (one vertex, no edges). The other 7 cases are treated schematically below.

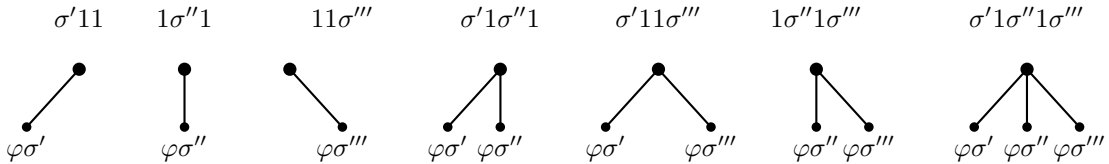
\begin{center}

\begin{pspicture}(-7.5,-1.6)(10,.8)
\psset{unit=.7cm}
\psdots(-10,-1.1)(-9,0)(-7,-1.1)(-7,0)(-5,0)(-4,-1.1)(-2,-1.1)(-1,0)(-1,-1.1)(1,-1.1)(2,0)(3,-1.1)(5,0)(5,-1.1)(6,-1.1)(8,-1.1)(9,0)(9,-1.1)(10,-1.1)

\psline(-10,-1.1)(-9,0)
\psline(-7,0)(-7,-1.1)
\psline(-5,0)(-4,-1.1)

\psline(-2,-1.1)(-1,0)(-1,-1.1)
\psline(1,-1.1)(2,0)(3,-1.1)
\psline(5,-1.1)(5,0)(6,-1.1)

\psline(8,-1.1)(9,0)(9,-1.1)
\psline(9,0)(10,-1.1)

\rput(-10,-1.51){\textrm{{\normalsize $\varphi \si'$}}}

\rput(-7,-1.51){\textrm{{\normalsize $\varphi \si''$}}}

\rput(-4,-1.51){\textrm{{\normalsize $\varphi \si'''$}}}

\rput(-2,-1.51){\textrm{{\normalsize $\varphi \si'$}}}
\rput(-1,-1.51){\textrm{{\normalsize $\varphi \si''$}}}

\rput(1,-1.51){\textrm{{\normalsize $\varphi \si'$}}}
\rput(3,-1.51){\textrm{{\normalsize $\varphi \si'''$}}}

\rput(5,-1.51){\textrm{{\normalsize $\varphi \si''$}}}
\rput(6,-1.51){\textrm{{\normalsize $\varphi \si'''$}}}

\rput(8,-1.51){\textrm{{\normalsize $\varphi \si'$}}}
\rput(9,-1.51){\textrm{{\normalsize $\varphi \si''$}}}
\rput(10,-1.51){\textrm{{\normalsize $\varphi \si'''$}}}

\rput(-9,1){\textrm{{\normalsize $\si' 1 1 $}}}
\rput(-7,1){\textrm{{\normalsize $1 \si'' 1 $}}}
\rput(-4,1){\textrm{{\normalsize $1 1 \si''' $}}}
\rput(-1,1){\textrm{{\normalsize $\si' 1 \si'' 1 $}}}
\rput(2,1){\textrm{{\normalsize $\si' 1 1\si'''  $}}}
\rput(5,1){\textrm{{\normalsize $1 \si'' 1 \si''' $}}}
\rput(9,1){\textrm{{\normalsize $\si' 1 \si'' 1 \si'''$}}}

\psset{dotsize=5pt 0}
\psdots(-9,0)(-7,0)(-5,0)(-1,0)(2,0)(5,0)(9,0)

\end{pspicture}
\end{center}
It is clear, by induction, that $\varphi$ is a bijection. Now let $A,P,D$ denote the statistics that count augmented ascents, plateaus, and augmented descents respectively in a Stirling permutation, and let $L,V,R$ denote the statistics that count left, vertical, and right edges respectively in a ternary tree. For $\si \in \q_n(213)$ and $\tau=\varphi(\si)$, it is easy to show by induction that
\[
L(\tau)=n-A(\si), \quad V(\tau)=n-P(\si), \quad R(\tau)=n-D(\si)\, .
\]
(For the base case $n=1$, $A,P,D$ all have the value 1 on $11 \in \q_{1}$ and $L,V,R$ all have the value 0 on the empty tree.)
Clearly, $L,V,R$ have a symmetric joint distribution on $\t_{n-1}$. Hence, $A,P,D$ likewise
have a symmetric joint distribution on $\q_n(213)$.


\subsection{The case $123$} \label{comb123} To explain the symmetry of $qC_{n,123}(p, q, r)$
in $p$ and $q$, we give a bijection from $\q_{n}(123)$ to
a suitable set $\a_n$, together with an involution on $\a_n$ that obviously interchanges the statistics corresponding to ``number of augmented descents'' and ``number of plateaus''.

A permutation $p \in \s_n$ determines a composition $c(p)$ of $n$: the distances between successive left-to-right (LR for short) minima in $p0$ ($=p$ with an appended 0). A composition $c=(c_1,c_2,\dots,c_k)$ determines a set of integer sequences $S(c):=\{(s_1,s_2,\dots,s_k):\, 1\le s_i \le c_i \textrm{ for all } i\}$. Set $\a_n=\{(p,s):\, p \in \s_n(123),\ s \in S(c(p))\}$. There is an obvious involution on $\a_n$:
$(p,s) \mapsto (p,c(p)+1-s)$. For example, $p=(4,6,5,2,1,3)$ has LR minima 4,2,1 and $c(p)=(3,1,2)$ and the involution sends $\big(p,(3,1,1)\big)$ to $\big(p,(1,1,2)\big)$.

A Stirling permutation $\si \in \q_n$ determines a permutation $p(\si) \in \s_n$ given by the first occurrences of the letters in $\si$.

Now we define a mapping $\psi:\ \q_{n} \mapsto \a_n$. Given $\si \in \q_{n}$, let $m_1,\dots,m_k$ denote the successive LR minima in $p(\si)$, and let $s_i$ be the number of distinct letters in the subword of $\si$ bounded by the two occurrences of $m_i$. Set $s=(s_1,s_2,\dots,s_k)$, and $\psi(\si) = (p(\si),s)$. Then
the restriction $\psi|_{\q_{n}(123)}$ is the desired bijection from $\q_{n}(123)$ to $\a_n$.

To show this works, let us consider an example. Let $\si \in \q_{n}(123)$ and so, consequently, $p(\si) \in \mathfrak{S}_{n}(123)$, and suppose $p(\si)=$
\[
11 \ 12 \quad \  7 \ 10 \ 9 \quad \  4 \quad \  3 \quad \  1 \ 8 \ 6 \ 5 \ 2\, ,
\]
where we have inserted some space before each LR minimum. The spaces divide
$p(\si)$ into segments whose lengths form $c(p(\si))$. Since $p(\si)$ avoids 123,
the non-initial entries of all the segments are decreasing left to right. The Stirling property then forces
a plateau at each non-initial entry of a long segment (length $\ge 2$) and at each short segment
(length $= 1$):
\[
11 \ 12 \ 12 \quad \  7 \ 10 \ 10 \ 9 \ 9 \quad \  4 \ 4 \quad \  3 \ 3 \quad \  1 \ 8 \ 8 \ 6 \ 6 \ 5 \ 5 \ 2 \ 2
\]
As for each initial entry (\,= LR minimum) $m$, its second appearance must occur in its own segment (otherwise, $m\dots x \dots m$ appears with $x<m$) and it cannot split a plateau, but is otherwise unrestricted. Thus, for example, the second 7 may occur right after the first 7 (and 77 contains 1 distinct entry) or after the last 10 (and 7 10 10 7 contains 2 distinct entries) or after the last 9 (and 7 10 10 9 9 7 contains 3 distinct entries). In general, the number of choices to place the second occurrence of a LR minimum $m_i$ is the length $c_i$ of its segment. The validity of the bijection is now clear.

Next, there is a plateau at each short segment, at each non-initial entry in a long segment
and for each instance of $s_i=1$ (which means $m_i$ contributes a plateau).
So the number of plateaus corresponds to $ n -\#$ segments + $\v \{i:\ s_i=1\}\v$.
Similarly, there is an augmented descent after the plateau generated by each
short segment, after the plateau generated by each non-initial entry in a long segment and for each
instance of $s_i=c_i$ (which means the second occurrence of $m_i$ starts an additional augmented descent). So the number of augmented descents corresponds to $ n - \#$  segments + $\v \{i:\ s_i=c_i \}\v$.
The involution on $\a_n$ clearly interchanges these statistics.
\subsection{A further bijection} We now use $\a_n$ as an intermediate construct to give a bijection from $\q_n(123)$ to
a more appealing class of objects, denoted $\fc_n$, which we now define. A \emph{favorite-child} (FC) \emph{ordered tree} is an
(unlabeled) ordered tree in which each parent (non-leaf) vertex has a distinguished child edge or,
more picturesquely, a designated favorite child. Let $\fc_n$ denote the set of $n$-edge FC ordered trees.
It is convenient to introduce what we call the \emph{left-path} labeling of the vertices in an ordered tree,
defined recursively as follows.
\begin{itemize}
\item Place label 0 on the root.

\item Take the smallest labeled vertex $v$ with an unlabeled child (initially $v=0$). Successively
label the vertices in the leftmost path from each unlabeled child of $v$ (taken left to right) with the smallest unused label.

\item Repeat until all vertices are labeled.

\end{itemize}

\begin{center}

\begin{pspicture}(-3,-1.8)(4,5.5)
\psset{xunit=.7cm}
\psset{yunit=.5cm}
\psdots(-4,2)(-2,2)(0,0)(0,2)(2,2)(4,2)(-2,4)(-2,6)(1,4)(3,4)(1,6)(3,6)(0,8)(1,8)(2,8)(1.5,10)(2.5,10)
\psline(-4,2)(0,0)(-2,2)(-2,4)(-2,6)
\psline(0,2)(0,0)(2,2)(1,4)(1,6)(0,8)
\psline(1,8)(1,6)(2,8)(1.5,10)
\psline(0,0)(4,2)
\psline(2,8)(2.5,10)
\psline(2,2)(3,4)(3,6)

\rput(0,-0.5){0}
\rput(0,-0.5){0}
\rput(-4,2.5){1}
\rput(-2.4,2){2}
\rput(-2.4,4){3}
\rput(-2.4,6  ){4}
\rput(-.4,2  ){5}
\rput(2.3, 1.9 ){6}
\rput(0.7,4  ){7}
\rput(0.7,5.8  ){8}
\rput(-.3,8.2){9}
\rput(4, 2.5 ){10}

\rput(0,-2){ left path labeling -- first pass}
\rput(0,-3.3){ a)}
\psset{dotsize=5pt 0}
\psdots(0,0)

\end{pspicture}
\begin{pspicture}(-3,-1.8)(3,5.5)
\psset{xunit=.7cm}
\psset{yunit=.5cm}
\psdots(-4,2)(-2,2)(0,0)(0,2)(2,2)(4,2)(-2,4)(-2,6)(1,4)(3,4)(1,6)(3,6)(0,8)(1,8)(2,8)(1.5,10)(2.5,10)
\psline(-4,2)(0,0)(-2,2)(-2,4)(-2,6)
\psline(0,2)(0,0)(2,2)(1,4)(1,6)(0,8)
\psline(1,8)(1,6)(2,8)(1.5,10)
\psline(0,0)(4,2)
\psline(2,8)(2.5,10)
\psline(2,2)(3,4)(3,6)

\rput(0,-0.5){0}
\rput(0,-0.5){0}
\rput(-4,2.5){1}
\rput(-2.4,2){2}
\rput(-2.4,4){3}
\rput(-2.4,6  ){4}
\rput(-.4,2  ){5}
\rput(2.3, 1.9 ){6}
\rput(0.7,4  ){7}
\rput(0.7,5.8  ){8}
\rput(-.3,8.2){9}
\rput(4, 2.5 ){10}

\rput(3.4,4.2){11}
\rput(3.4,6.2){12}
\rput(1,8.5){13}
\rput(2.4,8  ){14}
\rput(1.5,10.5  ){15}
\rput(2.5,10.5  ){16}

\rput(0,-2){ final result}
\rput(0,-3.3){ b)}
\psset{dotsize=5pt 0}
\psdots(0,0)

\end{pspicture}

\end{center}
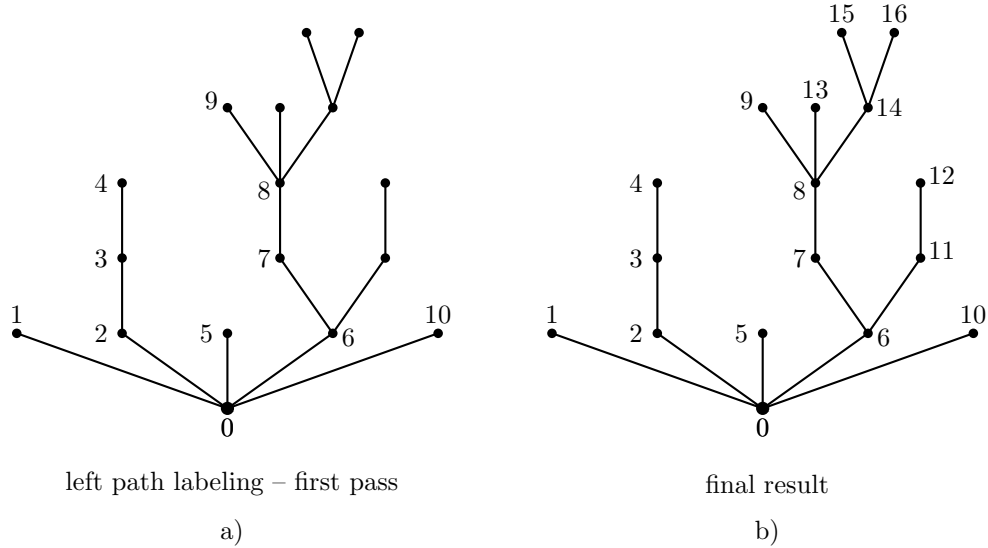
\centerline{Figure 1}

\vspace*{4mm}

\noindent For the ordered tree pictured in Figure 1 above, the labels generated from $v=0$ are shown on the left, the second pass uses $v=6$, and the full left-path labeling is shown on the right.

There are several known bijections from 321-avoiding permutations to Dyck paths,
equivalently, under reversal of permutations and the ``glove'' identification of Dyck paths and ordered trees, from 123-avoiding permutations to ordered trees.
(See \cite{Callan07,Claesson08} for two surveys of these bijections.)
Here, though, we need an apparently new one. Define $\rho:\, \msn(123) \mapsto \o_n$, the set of $n$-edge
ordered trees, as follows. Given $p \in \msn(123)$, split $p$ into segments, each starting at a
LR minimum of $p$. Form a tree on the vertex set $[0,n]$ by, for each segment, joining all its
entries to $m-1$ where $m$ is the first entry of the segment, as illustrated by example below
(the LR minimum segments are underlined).
\[
\begin{matrix}
\ \underbar{15\ 16}\ & \ \underbar{12}\ & \ \underbar{9\ 14\ 13}\ & \ \underbar{8}\ & \ \underbar{7\ 11}\ & \ \underbar{4}\ & \  \underbar{3}\ & \ \underbar{1\ 10\ 6\ 5\ 2}\\
\downarrow & \downarrow  & \downarrow & \downarrow & \downarrow & \downarrow & \downarrow & \downarrow \\
14 & 11 & 8 & 7 & 6 & 3 & 2 & 0
\end{matrix}
\]
These edges clearly form a tree; root it at 0. Then order the edges so that the children of each parent vertex are increasing left to right. (The result for this example is the tree shown in Figure 1b). Finally, erase all the labels to get the desired ordered tree.
To reverse the map, label the vertices of the tree in left-path order. The LR minima can then be retrieved: take the leftmost child of each   parent vertex. The length of the segment containing a LR minimum $v$ can also easily be retrieved as the family size (number of children) of the parent of $v$.  A 123-avoiding permutation is determined by its LR minima and their locations (all other entries decrease left to right), and so the original permutation can be recovered.

The efficacy of this bijection is that it takes the lengths of the LR minimum segments (visited right to left) to the family sizes of the parent vertices (visited in left-path order). A bijection from $\a_n$ to $\fc_n$ is now clear: for $(p,s) \in \a_n$, use $\rho(p)$ as the underlying ordered tree and use $s$ to designate the favorite child of each parent vertex. The involution on $\a_n$ that establishes the equidistribution of descents and plateaus in $\q_n(123)$ then becomes ``reverse the age ranking of each favorite child'',
i.e., change it from $i$-th (say) oldest to $i$-th youngest.

\subsection{Cases $123$ and $132$}
To see why plateaus have the same distribution on $\q_{n}(123)$ and $\q_{n}(132)$, observe that,
by considerations entirely analogous to those for the map $\psi|_{\q_{n}(123)}$ in Section \ref{comb123},
$\psi|_{\q_{n}(132)}$ is also a bijection, this time from $\q_{n}(132)$ to $\a_n$, and it also
carries ``number of plateaus'' to $ n -\#$ segments + $\v \{i:\ s_i=1\}\v$.

\section{Further results}
In this section we consider Stirling permutations that avoid 213 and another pattern (motivated by the study of  avoiding two patterns $132,\tau$ in permutations, see \cite{MV1} and references therein).
Let $\mqn(\tau_1,\tau_2)$ denote the set of Stirling permutations of order $n$ that avoid the patterns $\tau_1$ and $\tau_2$. For a pattern $\tau$, we define $$F_\tau=F_\tau(x,p,q,r)=\sum_{n\geq0}x^n\sum_{\sigma\in \mqn(213, \tau)}p^{\plat(\sigma)}q^{\des(\sigma)}r^{\asc(\sigma)}.$$
For patterns $\tau=(\tau_1,\dots,\tau_k)$ and $\tau'=(\tau'_1,\dots,\tau'_{k'})$, let $\tau \oplus \tau'$ denote their ``disjoint  concatenation'' $(\tau_1,\dots,\tau_k,m+\tau'_1,\dots,m+\tau'_{k'})$, where $m$ is the largest letter of $\tau$. Thus $11\oplus 121=11232$.

\begin{theorem}\label{thfur1}
Let $\tau=1\oplus \tau'$ where $\tau'$ is some pattern. Then, the generating function $F_\tau(x,p,q,r)$ is given by
\begin{align*}
F_\tau(x,p,q,r)=1+\frac{xp+xr(p+q)(F_{\tau'}(x,p,q,r)-1)+xqr^2(F_{\tau'}(x,p,q,r)-1)^2}
{1-xpq-xqr(1+p)(F_{\tau'}(x,p,q,r)-1)-xq^2r^2(F_{\tau'}(x,p,q,r)-1)^2}.
\end{align*}
\end{theorem}
\begin{proof}
Let us write an equation for the generating function $F_{\tau}(x,p,q,r)$.
Note that each nonempty Stirling permutation $\sigma$ that avoids both $213$ and $\tau$ can be represented as $\sigma'1\sigma''1\sigma'''$ such that
\begin{itemize}
\item each letter of $\sigma'$ is greater than each letter of $\sigma''$;
\item each letter of $\sigma''$ is greater than each letter of $\sigma'''$;
\item $\sigma'$ is a Stirling permutation that avoids both $213$ and $\tau$;
\item $\sigma'',\sigma'''$ are Stirling permutations that avoid both $213$ and $\tau'$.
\end{itemize}
Hence, by considering the $8$ possibilities of either one of $\sigma',\sigma'',\sigma'''$ is empty or not, we obtain that the generating function $F_{\tau}(x,p,q,r)$ satisfies
\begin{align*}
F_\tau(x,p,q,r)&=1+xp+xpq(F_\tau-1)+x(p+q)r(F_{\tau'}-1)+xqr(1+p)(F_\tau-1)(F_{\tau'}-1)\\
&+xr^2q(F_{\tau'}-1)^2+xq^2r^2(F_\tau-1)(F_{\tau'}-1)^2,
\end{align*}
which, by solving for $F_\tau(x,p,q,r)-1$, implies the required result.
\end{proof}

\begin{example}
Let $\tau=122 = 1\oplus \tau'$ with $\tau'=11$. Clearly, $F_{\tau'}=1$ since it is very difficult for a Stirling permutation to avoid a repeated letter. Thus, Theorem~\ref{thfur1} gives
\begin{align*}
F_{122}=F_{122}(x,p,q,r)&=1+\frac{xp}{1-xpq}=1+\sum_{j\geq0}x^{j+1}p^{j+1}q^j=1+\sum_{j\geq1}x^jp^jq^{j-1}.
\end{align*}
For $\tau=1233=1 \oplus 122$, Theorem~\ref{thfur1} gives
\begin{align*}
F_{1233}-1&=\frac{xp+xr(p+q)(F_{122}-1)+xqr^2(F_{122}-1)^2}
{1-xpq-xqr(1+p)(F_{122}-1)-xq^2r^2(F_{122}-1)^2}.
\end{align*}
In particular, $F_{1233}(x,1,1,1)=\frac{(1-x)^2}{1-3x+x^2}$, that is, the number of Stirling permutations of $\mqn(213,1233)$ is given by the $2n$-th Fibonacci number (the $n$-th Fibonacci number is defined by $a_0=0$, $a_1=1$ and $a_n=a_{n-1}+a_{n-2}$).
Applying Theorem~\ref{thfur1} repeatedly, we obtain
\begin{align*}
F_{12344}(x,1,1,1)&=\frac{(1-3x+x^2)^2}{1-7x+15x^2-12x^3+5x^4-x^5},\\
F_{123455}(x,1,1,1)&=\frac{(1-7x+15x^2-12x^3+5x^4-x^5)^2}{(1-x)(1-14x+77x^2-215x^3+332x^4-295x^5+157x^6-51x^7+10x^8-x^9)}.
\end{align*}
\end{example}

By Theorem~\ref{thfur1}, we obtain that the generating function $F_{123\cdots k(k+1)(k+1)}(x,p,q,r)$ is a rational function.

\begin{theorem}\label{thfur2}
Let $\tau=11 \oplus \tau'$ where $\tau'$ is some pattern. Then, the generating function $F_\tau(x,p,q,r)$ is given by
\begin{align*}
F_\tau=\frac{-b-\sqrt{b^2-4ac}}{2a},
\end{align*}
where
\begin{align*}
a&=qrx(1+qr(F_{\tau'}-1)),\\
b&=-1-qx(r-p)-qrx(2qr-p-r)(F_{\tau'}-1),\\
c&=1+xp(1-q)+rx(p+q^2r-qp-qr)(F_{\tau'}-1).
\end{align*}
\end{theorem}
\begin{proof}
Let us write an equation for the generating function $F_{\tau}(x,p,q,r)$.
Note that each nonempty Stirling permutation $\sigma$ that avoids both $213$ and $\tau$ can be represented as $\sigma'1\sigma''1\sigma'''$ such that
\begin{itemize}
\item each letter of $\sigma'$ is greater than each letter of $\sigma''$;
\item each letter of $\sigma''$ is greater than each letter of $\sigma'''$;
\item $\sigma',\sigma''$ is a Stirling permutation that avoids both $213$ and $\tau$;
\item $\sigma'''$ are Stirling permutations that avoid both $213$ and $\tau'$.
\end{itemize}
Hence, by considering the $8$ possibilities of either one of $\sigma',\sigma'',\sigma'''$ is empty or not, we obtain that the generating function $F_{\tau}(x,p,q,r)$ satisfies
\begin{align*}
F_\tau(x,p,q,r)&=1+xp+x(p+r)q(F_\tau-1)+xpr(F_{\tau'}-1)+xqr(F_\tau-1)^2\\
&+xqr(p+r)(F_{\tau}-1)(F_{\tau'}-1)+xq^2r^2(F_{\tau'}-1)(F_{\tau}-1)^2,
\end{align*}
which, by solving for $F_\tau(x,p,q,r)-1$, implies the required result.
\end{proof}

\begin{example}
Let $\tau=1122 = 11 \oplus 11$. Since $F_{11}=1$, Theorem \ref{thfur2} gives
$$F_{1122}(x,p,q,r)=\frac{1-xqp+xqr-\sqrt{x^2q^2p^2-2xqp+2x^2q^2pr+1-2xqr+x^2q^2r^2-4x^2qrp}}{2qrx},$$
which leads to $F_{1122}(x,1,1,1)=C(x)$,  where $C(x)=\frac{1-\sqrt{1-4x}}{2x}$ is the generating function for the Catalan numbers.

Using Theorem \ref{thfur2} once more, we have
$$F_{112233}(x,1,1,1)=\frac{1-\sqrt{2\sqrt{1-4x}-1}}{1-\sqrt{1-4x}}=C(xC(x))$$
and
$$F_{11223344}(x,1,1,1)=C(xC(xC(x))).$$
By induction on $k$, we obtain that $F_{1122\cdots kk}(x,1,1,1)=C(xC(xC(x\cdots C(xC(x)))))$, where $C$ is used exactly $k-1$ times.
\end{example}

As a final example, let us count the occurrences of the pattern $122$ in $\mqn(213)$ (motivated by the study of counting occurrences of the pattern $12\cdots k$ in a $132$-avoiding permutation, for example see \cite{MV0,RZW}). To do so, we denote the number occurrences of the pattern $122$ in $\sigma$ by $122(\sigma)$. We define $R(x,p,z)$ to be the generating function for the number of Stirling permutations of $\mqn(213)$ according to the occurrences of plateaus and occurrences of the pattern $122$, namely,
$$R(x,p,z)=\sum_{n\geq0}x^n\sum_{\sigma\in\mqn(213)}p^{\plat(\sigma)}z^{122(\sigma)}.$$
By the $8$ possibilities of block decompositions in the proof of Theorem \ref{thfur1}, we obtain
$$R(x,p,z)=1+xpR(x,p,z)R(x,pz^2,z)+xR(x,p,z)(R(x,pz,z)-1)R(x,pz^2,z),$$
which implies
$$R(x,p,z)=\frac{1}{1-x(R(x,pz,z)-1+p)R(x,pz^2,z)}.$$
The first terms of the generating function $R(x,p,z)$ are $1$, $px$,  $p({p}{z}^{2}+z+{p})$,
$p( {p}^{2}{z}^{6}+2\,p{z}^{3}+{p}^{2}{z}^{4}+p{z}^{4}+{z}^{2
}+p{z}^{2}+2\,{p}^{2}{z}^{2}+2\,pz+{p}^{2})$ and  $p( {p
}^{3}+7\,{p}^{2}{z}^{3}+3\,p{z}^{2}+2\,p{z}^{3}+2\,{p}^{2}{z}^{2}+3\,{
p}^{3}{z}^{4}+3\,{p}^{3}{z}^{2}+3\,{p}^{2}z+3\,{p}^{2}{z}^{4}+3\,{p}^{
3}{z}^{6}+4\,p{z}^{4}+p{z}^{6}+{z}^{3}+2\,{p}^{2}{z}^{6}+4\,{p}^{2}{z}
^{7}+5\,{p}^{2}{z}^{5}+{p}^{3}{z}^{10}+2\,{p}^{3}{z}^{8}+{p}^{2}{z}^{8
}+2\,p{z}^{5}+{p}^{2}{z}^{9}+{p}^{3}{z}^{12})$.

\section*{Acknowledgements.}
The second author was supported by NSFC (11401083) and the Fundamental Research Funds for the Central Universities (N152304006).
The authors would like to thank referees for their valuable suggestions.

\end{document}